\documentclass[notitlepage,leqno,10pt]{article}
\usepackage{latexsym}
\usepackage{color} 
\usepackage{amsmath}
\usepackage{amsthm}
\usepackage{amsfonts}
\usepackage{amssymb}
\usepackage{graphicx}
\usepackage{enumerate}
\renewcommand{\a }{\alpha }
\renewcommand{\b }{\beta }
\renewcommand{\d}{\delta }

\newcommand{\e }{\epsilon }
\newcommand{\ve}{\varepsilon}
\newcommand{\gve}{{g_{\varepsilon}}}
\newcommand{\g }{\gamma}

\renewcommand{\l }{\lambda }

\newcommand{\m }{\mu }
\newcommand{\n }{\nabla }

\newcommand{\Sp}{\mathbb{S}}

\renewcommand{\div }{{\rm{div}} }

\newcommand{\intbar}{\mathop{\int\makebox(-13.5,0){\rule[4pt]{.7em}{0.3pt}}%
\kern-6pt}\nolimits}

\newcommand{\be}{\begin{equation}}
\newcommand{\ee}{\end{equation}}

\newcommand{\C}{\mathbb{C}}
\newcommand{\R}{\mathbb{R}}

\newcommand{\Rtre}{\mathbb{R}^3}

\newcommand{\N}{\mathbb{N}}

\newcommand{\eps}{\varepsilon}
\newcommand{\p}{\partial}

\def\bn{\vec{n}}
\def\bX{\vec{X}}
\def\bx{\vec{x}}
\def\bH{\vec{H}}
\DeclareMathOperator{\Area}{Area}

\DeclareMathOperator{\diam}{diam}

\DeclareMathOperator{\Scal}{Scal}
\DeclareMathOperator{\Ric}{Ric}
\DeclareMathOperator{\Id}{Id}

\begin{document}

\author{Paul Laurain$^{1}$, Andrea Mondino$^{2}$}

\date{}

\title{Concentration of small Willmore spheres  in Riemannian 3-manifolds}

\newtheorem{lem}{Lemma}[section]
\newtheorem{pro}[lem]{Proposition}
\newtheorem{thm}[lem]{Theorem}
\newtheorem{rem}[lem]{Remark}
\newtheorem{cor}[lem]{Corollary}
\newtheorem{df}[lem]{Definition}
\newtheorem{ex}[lem]{Example}
\newtheorem*{Theorem}{Theorem}
\newtheorem*{Lemma}{Lemma}
\newtheorem*{Proposition}{Proposition}
\newtheorem*{claim}{Claim}

\maketitle

\footnotetext[1]{Institut Math\'ematiques de Jussieu,Paris VII, B\'atiment Sophie Germain, Case 7012,
75205 PARIS Cedex 13, France. E-mail address: laurainp@math.jussieu.fr }

\footnotetext[2]{ETH, R\"amistrasse 101, Zurich, Switzerland. E-mail address: andrea.mondino@math.ethz.ch}

\

\

\begin{center}
\noindent {\sc abstract}. 
Given a 3-dimensional Riemannian manifold $(M,g)$, we  prove that  if $(\Phi_k)$ is a sequence of Willmore spheres (or more generally area-constrained Willmore spheres), having Willmore energy bounded above uniformly strictly  by $8 \pi$, and  Hausdorff converging to a point $\bar{p}\in M$, then $\Scal(\bar{p})=0$ and $\nabla \Scal(\bar{p})=0$ (resp. $\nabla \Scal(\bar{p})=0$). Moreover, a suitably rescaled sequence smoothly converges, up to subsequences and reparametrizations, to a round sphere in the euclidean $3$-dimensional space. 
\\This generalizes previous results  of Lamm and Metzger contained in \cite{LM1}-\cite{LM2}.
\\An application to the Hawking mass is also established.
\bigskip\bigskip

\bigskip

\centerline{\bf AMS subject classification: }
49Q10, 53C21, 53C42 , 35J60, 83C99.
\end{center}

\section{Introduction}\label{s:in}
Let  $\Sigma$ be  a closed two dimensional surface and $(M,g)$ a $3$-dimensional Riemannian manifold. Given a smooth immersion $\Phi:\Sigma \hookrightarrow M$, $W(\Phi)$ denotes the Willmore energy of $\Phi$ defined by 
\begin{equation}
W(\Phi):=\int_{\Sigma} H^2 \, dvol_{\bar{g}},
\end{equation}
where $\bar{g}:=\Phi^* (g)$ is the pullback metric on $\Sigma$ (i.e. the metric induced by the immersion), $dvol_{\bar{g}}$ is the associated volume form, and $H$ is the mean curvature of the immersion $\Phi$ (we adopt the convention that $H=\frac{1}{2} \bar{g}^{ij} A_{ij}$ where $A_{ij}$ is the second fundamental form; or, in other words, $H$ is the arithmetic mean of the two principal curvatures).
\\

In case the ambient manifold is the euclidean 3-dimensional space, the topic is classical and goes back to the works of Blaschke and Thomsen in 1920-'30 who were looking for a conformal invariant theory which included minimal surfaces; the functional was later rediscovered by Willmore \cite{Will} in the 60'ies and from that moment there have been a flourishing of results (let us mention the fundamental paper of Simon \cite{SiL}, the work of Kuwert-Sch\"atzle \cite{KS1}-\cite{KS}-\cite{KS2}, the more recent approach by Rivi\`ere \cite{Riv1}-\cite{Riv2}-\cite{RivPCMI}, etc.) culminated with the recent proof of the Willmore Conjecture by Marques and Neves \cite{MN} by min-max techniques (let us mention that partial results towards the Willmore conjecture were previously obtained by Li and Yau \cite{LY}, Montiel and Ros \cite{MonRos},  Ros \cite{Ros}, Topping \cite{Top}, etc., and that a crucial role in the proof of the conjecture is played by a result of Urbano \cite{Urb}).


On the other hand, the investigation of the Willmore functional in non constantly curved Riemannian manifolds is a much more recent topic started in \cite{Mon1} (see also \cite{Mon2} and the more recent joint work with Carlotto \cite{CM}) where the second author studied existence and non existence of Willmore surfaces in a perturbative setting.
\\Smooth minimizers of the $L^2$ norm of the second fundamental form among spheres in compact Riemannian three manifolds were obtained in collaboration with Kuwert and Schygulla in \cite{KMS} where the full regularity theory for minimizers was settled taking inspiration from the  approach of Simon \cite{SiL} (see also  \cite{MonSchy} for minimization in non compact Riemannian manifolds).
 \\Let us finally mention the work in collaboration with Rivi\`ere \cite{MR1}-\cite{MR2} where, using a ``parametric approach'' inspired by the Euclidean theory of  \cite{Riv1}-\cite{Riv2}-\cite{RivPCMI}, the necessary tools for studying the calculus of variations of the Willmore functional in Riemannian manifolds (i.e. the definition of the weak objects and related  compactness and regularity issues) are settled together with applications; in particular the existence and regularity of Willmore spheres in homotopy classes is established.
 \\

Since -as usual in the calculus of variations- the  existence results are obtained by quite general techniques and do not describe  the minimizing object, the purpose of the present paper is to  investigate the geometric properties of the critical points of $W$.  
\\More precisely we investigate the following natural questions: Let $\Phi_k: \Sp^2 \hookrightarrow M$ be a sequence of smooth critical points of the Willmore functional $W$ (or more generally we will also consider critical points under area constraint) converging  to a point $\bar{p}\in M$ in Hausdorff distance sense; what can we say about $\Phi_k$? are they becoming more and more round? Has the limit point $\bar{p}$ some special geometric property?
\\

These questions  have already been addressed in recent  articles -below the main known results are recalled by the reader's convenience-, but in the present paper we are going to obtain the sharp answers.
\\Before passing to describe the known and the new results in this direction, let us recall that a critical point of the Willmore functional is called \emph{Willmore surface} and it satisfies:
\begin{equation}\label{eq:PDEW}
\Delta_{\bar{g}} H + H |A^\circ|^2 + H \Ric(\bn, \bn)=0,
\end{equation}
where $\Delta_{\bar{g}}$ is the Laplace-Beltrami operator corresponding to the metric $\bar{g}$, $(A^\circ)_{ij}:=A_{ij}-H\bar{g}_{ij}$ is the trace-free second fundamental form, $\bn$ is a normal unit vector to $\Phi$, and $\Ric$ is the Ricci tensor of the ambient manifold $(M,g)$. Notice that \eqref{eq:PDEW} is a  fourth-order nonlinear elliptic PDE in the parametrization map $\Phi$.
\\Throughout the paper we will consider more generally \emph{area-constrained Willmore surfaces}, i.e. critical points of the  Willmore functional under area constraint; the immersion $\Phi$ is an area-constrained Willmore surface if and only if it satisfies
\begin{equation}\label{eq:PDEWC}
\Delta_{\bar{g}} H + H |A^\circ|^2 + H \Ric(\bn, \bn)=\lambda H,
\end{equation}
for some  $\lambda\in \R$ playing the role of Lagrange multiplier. 
\\

The first result in the direction of the above questions was achieved in the master degree thesis of the second author \cite{Mon1} where it was proved that if $(\Phi_k)$ is a sequence of Willmore surfaces obtained as normal graphs over shrinking geodesic spheres centered at a point $\bar{p}$, then the scalar curvature at $\bar{p}$ must vanish: $\Scal(\bar{p})=0$.

In the subsequent papers  \cite{LM1}-\cite{LM2},  Lamm and Metzger  proved that if $\Phi_k:\Sp^2 \hookrightarrow M$ is a sequence of area-constrained Willmore surfaces converging to a point $\bar{p}$ in Hausdorff distance sense and such that \footnote{notice that the normalization of the Willmore functional used in  \cite{LM1}-\cite{LM2} differ from our convention by a factor 2}
\be\label{eq:assLM}
W(\Phi_k) \leq 4\pi+ \varepsilon \quad \text{ for some $\varepsilon>0$ small enough},
\ee
then $\nabla \Scal(\bar{p})=0$ and, up to subsequences, $\Phi_k$ is $W^{2,2}$-asymptotic to a geodesic sphere centered at $\bar{p}$. Moreover in \cite{LM2}, using the regularity theory developed in \cite{KMS},  they showed that if $(M,g)$ is any compact Riemannian $3$-manifold and $a_k$ is any sequence of positive real numbers such that $a_k\downarrow 0$   then there exists a smooth minimizer $\Phi_k$ of $W$ under the area-constraint $\Area(\Phi_k)=a_k$; moreover such sequence $(\Phi_k)$ satisfies \eqref{eq:assLM} and therefore it $W^{2,2}$-converges to a round critical point of the scalar curvature. Let us mention that the existence of area-constrained Willmore spheres was generalized in \cite{MR2} to any value of the area.
\\

The goal of this paper is multiple. The main achievement is the improvement of the perturbative bound \eqref{eq:assLM} above to the global bound 
\be\label{eq:ass}
\limsup_{k} W(\Phi_k) < 8 \pi.
\ee
Secondly we  improve the $W^{2,2}$-convergence above to \emph{smooth} convergence towards a \emph{round} critical point of the scalar curvature, i.e. we show that if we rescale  $(M,g)$ around $\bar{p}$ in such a way that the sequence of surfaces has fixed area equal to one (for more details see Section \ref{Sec:NP}), then the sequence converges smoothly, up to subsequences, to a round sphere centered at $\bar{p}$, and $\bar{p}$ is a critical point of the scalar curvature of $(M,g)$.
\\Finally  we  give  an application of these results to the Hawking mass.  
\\We believe  that the bound \eqref{eq:ass} is sharp in order to have smooth convergence to a \emph{round} point (in the sense specified above);  indeed, if \eqref{eq:ass} is violated then  the sequence $(\Phi_k)$ may degenerate  to a couple of bubbles,  each one costing almost $4 \pi$ in terms of Willmore energy.  
\\

Now let us state the main results of the present article.  The first theorem below concerns the case of a sequence of Willmore immersions and it is a consequence of the second more general theorem about area-constrained Willmore immersions.
\begin{thm}\label{thm1}
Let $(M,g)$ be a 3-dimensional Riemannian manifold and let $\Phi_k: \Sp^2 \hookrightarrow M$ be a sequence of Willmore surfaces satisfying the energy bound  \eqref{eq:ass} and 
Hausdorff converging to a point $\bar{p}\in M$. 

Then $\Scal(\bar{p})=0$ and $\nabla \Scal(\bar{p})=0$; moreover, if we rescale $(M,g)$ around $\bar{p}$ in such a way that the rescaled immersions $\tilde{\Phi}_k$ have fixed area equal to one, then $\tilde{\Phi}_k$ converges smoothly, up to subsequences and up to reparametrizations, to a round sphere in the 3-dimensional euclidean space.
 \end{thm}

Actually we prove the following more general result about sequences of area-constrained Willmore immersions.

\begin{thm}\label{thm2}
Let $(M,g)$ be a 3-dimensional Riemannian manifold and let $\Phi_k: \Sp^2 \hookrightarrow M$ be a sequence of area-constrained Willmore surfaces satisfying the energy bound  \eqref{eq:ass} and 
Hausdorff converging to a point $\bar{p}\in M$. 

Then  $\nabla \Scal(\bar{p})=0$; moreover, if we rescale $(M,g)$ around $\bar{p}$ in such a way that the rescaled immersions $\tilde{\Phi}_k$ have fixed area equal to one, then $\tilde{\Phi}_k$ converges smoothly, up to subsequences and up to reparametrizations, to a round sphere in the 3-dimensional euclidean space.
\end{thm}

Of course Theorem \ref{thm2} implies Theorem \ref{thm1} except the property $\Scal(\bar{p})=0$. This fact follows by the aforementioned \cite[Theorem 1.3]{Mon1} holding  for Willmore  graphs over  geodesic spheres, together with the smooth convergence to a round point ensured by  Theorem \ref{thm2}.
\\

Now we pass to discuss an application to the Hawking mass $m_H$, defined for an immersed sphere $\Phi:\Sp^2 \hookrightarrow (M,g)$ by
\be\label{eq:defmH}
m_H(\Phi)=\frac{\Area_g(\Phi)}{16 \pi^{3/2}} \left(4\pi-W(\Phi)\right).
\ee
Of course,  the critical points of the Hawking mass under area constraint are exactly the  area-constrained Willmore spheres (see \cite{LMS} and the references therein for more material about the Hawking mass); moreover it is clear that the inequality $m_H(\Phi)\geq 0$ implies that $W(\Phi)\leq 4\pi$. 
\\Therefore, combining this easy observations with Theorem \ref{thm2}, we obtain the following corollary.

\begin{cor}\label{cor:Haw}
Let $(M,g)$ be a 3-dimensional Riemannian manifold and let $\Phi_k: \Sp^2 \hookrightarrow M$ be a sequence of critical points of $m_H$ under area constraint having non negative Hawking mass and 
Hausdorff converging to a point $\bar{p}\in M$. 
 
Then  $\nabla \Scal(\bar{p})=0$; moreover, if we rescale $(M,g)$ around $\bar{p}$ in such a way that the rescaled immersions $\tilde{\Phi}_k$ have fixed area equal to one, then $\tilde{\Phi}_k$ converges smoothly, up to subsequences and up to riparametrizations, to a round sphere in the 3-dimensional euclidean space.
\end{cor}

Let us briefly  comment on the relevance of   Corollary \ref{cor:Haw} despite the triviality of its proof.   Recall that, from the note of Christodoulou and Yau \cite{CY},  if $(M,g)$ has non negative scalar curvature then isoperimetric spheres (and more generally stable CMC spheres)  have positive Hawking mass; on the other hand it is known (see for instance \cite{Druet} or \cite{NardAGA}) that,  if $M$ is compact, then small  isoperimetric regions converge to geodesic spheres centered at a maximum point of the scalar curvature as the enclosed volume converges to $0$ (see also \cite{MonNard} for the non-compact case). Therefore a link between regions with positive Hawking mass and critical points of the scalar curvature  was already present in literature, but Corollary \ref{cor:Haw} expresses this link precisely.   
\\

We end the introduction by outlying the structure of the paper and the main ideas of the proof. First of all, as already noticed, it is enough to prove Theorem \ref{thm2} in order to get all the stated results. To prove it, we adopt the blow up technique taking inspiration from  \cite{Lau1} where the first author analyzed the corresponding questions  in the context of CMC-surfaces;  such technique was introduced in the analysis of the Yamabe problem which is a second order  scalar problem (for an detailed overview of the method including applications see \cite{DHR}), the technical novelty of \cite{Lau1} was that that a second order \emph{vectorial} problem was considered; the technical  originality of the  present paper from the point of view of the blow up method  is that  we study a \emph{fourth order vectorial problem}.

More precisely, in Section \ref{Sec:NP} we consider normal coordinated centered at the limit point $\bar{p}$ and we rescale appropriately the metric $g$ such that the rescaled surfaces have all diameter one (or, thanks to the monotonicity formula, it is equivalent  to fix the area of the rescaled surfaces equal to one); notice that the rescaled ambient metrics $g_k$ are  becoming more and more euclidean.
\\In Subsection \ref{subsec:AreaConstrW}, by exploiting the divergence form of the Willmore equation established in \cite{MR2}, we give a decay estimate on the Lagrange multipliers as $k$ goes to infinity.

Section \ref{Sec:3} is devoted to the proof of Theorem \ref{thm2}; we start in Subsection \ref{SubSec:3.1} by establishing a fundamental technical result telling that, under the above working assumptions, the sequence $(\Phi_k)$ converges smoothly to a round sphere, up to subsequences and  reparametrizations. Let us remark that in the proof we exploit in a   crucial way the assumption \eqref{eq:ass}, otherwise it may be possible for the sequence to degenerate to a couple of bubbles. Once we have smooth convergence to a round sphere $\omega$, we study the remainder given by the difference between $\Phi_k$ and $\omega$: in Subsection \ref{approx} we use the linearized Willmore operator (recalled in Appendix \ref{aL}) in order to give  precise asymptotics of such remainder term and in the final Subsection \ref{subsec:3.4} we refine these  estimates and conclude the proof.

\subsection{Acknowledgment}
The authors acknowledge the partial support of CNRS and Institut Math\'ematiques de Jussieu which made possible a visit of A.M. to Paris where the project was started. A.M. is supported by  the ETH fellowship.   

\section{Notation and preliminaries}\label{Sec:NP}
Throuhout the paper $(M,g)$ is a Riemannian 3-manifold and  $\Sp^2$ is the round $2$-sphere of unit radius in $\R^3$. The greek indexes $\a,\b,\g, \m,\nu$ will run from 1 to 3 and will denote quantities in $M$, latin indexes will run from 1 to 2 and will denote quantities on $\Phi_k(\Sp^2)$, we will always use Einstein notation on summation over indexes. 
Given a smooth immersion $\Phi:\Sp^2 \hookrightarrow (M,g)$ we call  $\bar{g}=\Phi^*(g)$ the pullback metric, $dvol_{\bar{g}}$ the induced area form, $H_{g,\Phi}$ the mean curvature and
$$W_g(\Phi):= \int_{\Sp^2} |H_{g,\Phi}|^2 dvol_{\bar{g}} $$
is the Willmore functional.


Now let $(\Phi_k)$ be a sequence of smooth immersions from  $\Sp^2$ into $M$. Under our working assumptions, called $\diam_g(\Omega)$  the diameter of the subset $\Omega$ of  $M$ with respect to the metric $g$, we will always have 
\begin{eqnarray} 
\ve_k&:=&\diam_g(\Phi_k(\Sp^2)) \to 0 \label{eq:diamto0},\\
W_g(\Phi_k)&:=& \int_{\Sp^2} |H_{g,\Phi_k}|^2 dvol_{\bar{g}_k}\leq 8 \pi-2\delta,\quad \text{for some } \delta>0 \text{ independent of } k \label{eq:HpWg}
\end{eqnarray}
where $dvol_{\bar{g}_k}$ is the area form on $\Sp^2$ associated to the pullback  metric $\bar{g}_k=\Phi_k^*(g)$ and $H_{g,\Phi_k}$ is the mean curvature of $\Phi_k$. 

Notice that in case  $M$ is compact then  \eqref{eq:diamto0} is sufficient to ensure that, up to subsequences, $\Phi_k(\Sp^2)$ converges to a point $\bar{p} \in M$ in Hausdorff distance sense; but since there is no further reason to restrict to a compact ambient manifold we assume the  convergence to $\bar{p}$ in the hyphotesis of our main results instead of a compactness assumption on $M$. 

In order to efficiently handle the geometric quantities we need good coordinates; let us now introduce them. Take coordinates $(x^\mu), \mu=1,2,3$ around $\bar{p}$ and let $p_k=(p^1_k,p^2_k,p^3_k)$ be the center of mass of $\Phi_k(\Sp^2)$:
$$p_k^\mu=\frac{1}{\Area_g(\Phi_k)} \int_{\Sp^2} \Phi^\mu_k dvol_{\bar{g}_k},\quad \mu=1,2,3, $$
where  $\Area_g(\Phi_k)= \int_{\Sp^2} dvol_{\bar{g}_k}$ is the area of $\Phi_k(\Sp^2)$.
 Clearly, up to subsequences, $p_k \to \bar{p}$.
 \\ For every $k\in \N$ consider the exponential normal coordinates centered in $p_k$ and  rescale this chart by a factor $\frac{1}{\ve_k}$ with respect to the center of these coordinates. Hence we get a new sequence of immersions  $\tilde{\Phi}_k:\Sp^2 \hookrightarrow (\Rtre, g_{\ve_k})$, in the sequel simply denoted by $\Phi_k$, where the metric $g_{\ve_k}$ is defined by
\be\label{def:ge}
g_{\ve_k}(y)(u,v):=g(\ve_k y)(\ve_k^{-1}u, \ve_k^{-1} v).
\ee 
Notice that now we have
\be\label{eq:RescQuant}
W_{g_{\ve_k}} (\Phi_k)\leq 8\pi-2\delta, \quad \diam_{g_{\ve_k}}(\Phi_k(\Sp^2))=1 \quad \text{and} \quad  \Phi_k(\Sp^2)\subset B_{g_{\ve_k}}(0,3/2),
\ee
where the first inequality is a consequence of the invariance under rescaling of the Willmore functional, and $B_{g_{\ve_k}}(0,3/2)$ is the metric ball  in $(\Rtre, g_{\ve_k})$ of center $0$ and radius $3/2$. By the classical expression of the metric in normal coordinates, we get that (see Appendix B in \cite{Lau1})
\be\label{eq:expge}
(g_{\ve_k})_{\m \nu}(y)=\d_{\m \nu}+\frac{\ve_k^2}{3} \,R_{\a \m \nu \b}(p_k) \,y^\a y^\b+\frac{\ve_k^3}{6} \,R_{\a \m \nu \b,\g}(p_k)\, y^\a y^\b y^\g+o(\ve_k^3),
\ee
the inverse metric is
\be\label{eq:expgei}
(g_{\ve_k})^{\m \nu}(y)=\d_{\m \nu}-\frac{\ve_k^2}{3} \,R_{\a \m \nu \b}(p_k) \,y^\a y^\b-\frac{\ve_k^3}{6} \,R_{\a \m \nu \b,\g}(p_k)\, y^\a y^\b y^\g+o(\ve_k^3),
\ee
the volume form of $g_{\ve_k}$ on  can be written as 
\be\label{eq:detge}
\sqrt{|g_{\ve_k}}|(y)=1-\frac{\ve_k^2}{6} \Ric_{\a \b}(p_k) y^\a y^\b-\frac{\ve_k^3}{12} \Ric_{\a \b,\g}(p_k) y^\a y^\b y^\g +o(\ve_k^3),
\ee
and the Christoffel symbols of $g_{\ve_k}$ can be expanded as
\be\label{eq:Gammage}
(\Gamma_{\ve_k})^\g_{\a \b}(y)=A_{\a \b \g \m}(p_k) y^\m \ve_k^2+ B_{\a \b \g \m \nu} y^\m y^\nu \ve_k^3+o(\ve_k^3) 
\ee
where $A_{\a \b \g \m}(p_k)=\frac{1}{3}(R_{\b \m \a \g}(p_k)+R_{\a \m \b \g}(p_k))$  and
\\$B_{\a \b \g \m \nu}(p_k)=\frac{1}{12} (2 R_{\b \m \a \g, \nu}(p_k)+2 R_{\a \mu \b \g,\nu}(p_k)+ R_{\b \m \nu \g, \a}+R_{\a \mu \nu \g, \b}(p_k)-R_{\a \mu \nu \b, \g}(p_k)). $
\\

Since by \eqref{eq:expge} the metric $g_{\ve_k}$ is close to the euclidean metric  in $C^\infty$ norm on $B_{g_0}(0,2)$, where $B_{g_0}(0,2)$ is the euclidean ball in $\R^3$ of center $0$ and radius $2$, recalling  \eqref{eq:RescQuant} we get the following lemma.
\begin{lem}\label{lem:AreaBound}
Let $g_{\ve_k}$ be the metric  defined in \eqref{def:ge} having  the form \eqref{eq:expge}; let $\Phi_k:\Sp^2 \hookrightarrow (\R^3, g_{\ve_k})$ be smooth immersions  with $\Phi_k(\Sp^2)\subset B_{g_{\ve_k}}(0,2)$  satisfying 
$$W_{g_{\ve_k}}(\Phi_k)\leq 8\pi-2\d, \quad \text{for some } \delta>0.$$
Then, for $k$ large enough, we have  
\be\label{eq:Wge}
W_{g_0}(\Phi_k)\leq 8 \pi-\d, \quad \frac{1}{2}\leq \diam_{g_0} (\Phi_k(\Sp^2)) \leq 2 \quad \text{and} \quad \Phi_k(\Sp^2)\subset B_{g_0}(0,2),
\ee
where $g_0$ is the euclidean metric on $\R^3$, $W_{g_0}$ is the euclidean Willmore functional and $B_{g_0}(0,2)$ is the euclidean ball of center $0$ and radius $2$ in $\R^3$.
It follows that, for large $k$, $\Phi_k:\Sp^2 \hookrightarrow (\R^3, g_{\ve_k})$ is a smooth embedding and that there exist constants $C_1,C_2>0$ such that
\be\label{eq:AreaBound}
0< \frac{1}{C_1}\leq \frac{1}{C_2} \Area_{g_0}(\Phi_k)\leq \Area_{g_{\ve_k}}(\Phi_k)\leq {C_2} \Area_{g_0}(\Phi_k) \leq C_1 < \infty.
\ee
\end{lem}

\begin{proof}
The properties expressed in \eqref{eq:Wge} follow from \eqref{eq:RescQuant} by a direct estimate of the remainders given by the curvature terms of the metric $g_{\ve_k}$; for such estimates we refer to  Lemma 2.1, Lemma 2.2, Lemma 2.3 and Lemma 2.4 in \cite{MonSchy}. 
\\It is classically known that if the Willmore functional of an immersed closed surface in $(\R^3,g_0)$ is strictly below $8\pi$ then the immersion is actually an embedding (see \cite{LY} or  \cite{SiL}), so our second statement follows. 
\\In order to prove \eqref{eq:AreaBound} let us recall  Lemma 1.1 in \cite{SiL} stating that 
$$
\sqrt{\frac{\Area_{g_0}(\Phi_k)}{W_{g_0}(\Phi_k)}}\leq \diam_{g_0}{\Phi_k(\Sp^2)} \leq C \sqrt{\Area_{g_0}(\Phi_k){W_{g_0}(\Phi_k)}} \quad \text{for some universal } C>0,
$$
which, combined with the bound on $\diam_{g_0}(\Phi_k(\Sp^2))$ and $W_{g_0}(\Phi_k)$ expressed in \eqref{eq:Wge}, gives that there exists a constant $C_0>0$ such that
$$
0<\frac{1}{C_0}\leq Area_{g_0}(\Phi_k) \leq C_0 < \infty;
$$
the desired chain of inequalities \eqref{eq:AreaBound} follows then by estimating the  remainders as in Lemma 2.2 in \cite{MonSchy}.
\end{proof}

\subsection{The area-constrained Willmore equation and an estimate of the Lagrange multiplier } \label{subsec:AreaConstrW}
In the rest of the paper we will work with area-constrained Willmore immersions, i.e. critical points of the Willmore functional under the constraint that the area is fixed. If $\Phi:\Sp^2\hookrightarrow (M,g)$ is a smooth area-constraint Willmore immersion, then it satisfies the following PDE (see for instance Section 3 in \cite{LMS} for the derivation of the equation)
\be\label{eq:AreaConstrW}
\triangle_{\bar{g}}H_{g,\Phi}+H_{g,\Phi} |A^\circ_{g,\Phi}|_{\bar{g}}^2+H_{g,\Phi} \Ric_g(\bn_{g,\Phi}, \bn_{g,\Phi}) =\lambda H_{g,\Phi}
\ee
for some $\l \in \R$, where $\bn_{g,\Phi}$ is  a normal unit vector to $\Phi(\Sp^2)\subset(M,g)$, $(A^\circ_{g,\Phi})_{ij}$ is the traceless second fundamental form $(A^\circ_{g,\Phi})_{ij}=(A_{g,\Phi})_{ij}-\bar{g}_{ij} H_{g,\Phi}$ (of course $(A_{g,\Phi})_{ij}$ is the second fundamental form of $\Phi$ in $(M,g)$) and $|A^\circ_{g,\Phi}|_{\bar{g}}^2=\bar{g}^{ik} \bar{g}^{jl} (A^\circ_{g,\Phi})_{ij} (A^\circ_{g,\Phi})_{kl}$ is its norm with respect to the metric $\bar{g}=\Phi^*g$.
\\

Now let $(\Phi_k)$ be a sequence of smooth area-constrained Willmore immersions of $\Sp^2$ into $(M,g)$ satisfying \eqref{eq:diamto0}-\eqref{eq:HpWg}; perform the rescaling procedure described above and obtain the immersions $(\tilde{\Phi}_k)$ of $\Sp^2$ into $(\R^3,g_{\ve_k})$ (for simplicity denoted again with $\Phi_k$ from now on), where $g_{\ve_k}$ is defined in \eqref{def:ge}, satisfying \eqref{eq:RescQuant}. Since the Willmore functional is scale invariant, the rescaled surfaces are still area-constrained Willmore surfaces so they satisfy the following equation
\be\label{eq:AreaConstrWk}
\triangle_{\bar{g}_{\ve_k}}H_{g_{\ve_k},\Phi_k}+H_{g_{\ve_k},\Phi_k} \; |A^\circ_{g_{\ve_k},\Phi_k}|_{\bar{g}_{\ve_k}}^2+H_{g_{\ve_k},\Phi_k} \Ric_{g_{\ve_k}}(\bn_{g_{\ve_k},\Phi_k}, \bn_{g_{\ve_k},\Phi_k}) =\lambda_k H_{g_{\ve_k},\Phi_k}.
\ee
The first step in our arguments is to show that the Lagrange multipliers $\lambda_k$ are controlled by $\ve_k^2$. The idea for obtaining informations on the Lagrange multipliers, as in \cite{LM2}, is to use the invariance under rescaling of the Willmore functional. 

\begin{lem}\label{lem:lagrOek}
Let $(\Phi_k)$ be a sequence of smooth area-constrained Willmore immersions of $\Sp^2$ into $(\R^3,g_{\ve_k})$ where $g_{\ve_k}$ has the form \eqref{eq:expge} with  $\ve_k\to 0$, and $\Phi_k(\Sp^2)\subset B_{g_0}(0,2)$, the euclidean ball of center $0$ and radius $2$.

Then the Lagrange multipliers $\lambda_k$ appearing in \eqref{eq:AreaConstrWk} satisfies:
\be\label{eq:lambdae2}
\sup_{k\in \N} \frac{|\lambda_k|}{\ve_k^2}<\infty.
\ee
\end{lem}

\begin{proof}
Since  $(\Phi_k)$ are area-constrained Willmore immersions, for every variation vector field  $\bX$ on $\R^3$ we have that
\be\label{eq:deltaWArea}
\delta_{\bX} W_{g_{\ve_k}}(\Phi_k)=\lambda_k  \delta_{\bX} \Area_{g_{\ve_k}}(\Phi_k),
\ee
where $\delta_{\bX} W$ and $\delta_{\bX} \Area$ are the first variations of the Willmore and the Area functionals corresponding to the vector field $\bX$. Observe that the vector field corresponding to the dilations in $\R^3$ is the position vector field $\bx$, so the first variation of the euclidean Willmore functional in $\R^3$ with respect to $\bx$ is null: $\delta_{\bx} W_{g_0}=0$; on the other hand the first variation of  euclidean area with respect to the $\bx$ variation is easy to compute using the tangential divergence formula:
$$\delta_{\bx} \Area_{g_0}(\Phi)=-2\int_{\Sp^2} <\bH,\bx>_{g_0} dvol_{\bar{g_0}}=\int_{\Sp^2} \div_{\Phi,g_0} \bx \; dvol_{\bar{g_0}}=2 \Area_{g_0}(\Phi),$$ 
where $\div_{\Phi,g_0} $ is the tangential divergence on $\Phi(\Sp^2)$ with respect to the euclidean metric.  The two euclidean formulas give the well known fact that every area-constraint Willmore surface is actually a Willmore surface. 

In the present framework, the ambient metric $g_{\ve_k}$ is a perturbation of order $\ve_k^2$ of the euclidean metric $g_0$, so it is natural to expect that the Lagrange multiplier maybe does not vanish but at least is of order $\ve_k^2$. Let us prove it.
First of all, by the expansion of the Christoffel symbols \eqref{eq:Gammage} it follows that the covariant derivative in metric $g_{\ve_k}$ of the position vector field $\bx$ has the form
\be\label{nablax}
\nabla^{g_{\ve_k}} \vec{x}= \Id+O(\ve_k^2).
\ee
It follows that the tangential divergence of $\bx$ on $\Phi_k(\Sp^2)$ with respect of the metric $\bar{g}_k$ is $\div_{\Phi,g_{\ve_k}} \bx=2+O(\ve_k^2)$ and by the tangential divergence formula we obtain as before
$$
\delta_{\bx} \Area_{g_{\ve_k}}(\Phi)=-2\int_{\Sp^2} <\bH_{\Phi_k,g_{\ve_k}},\bx>_{g_{\ve_k}} dvol_{\bar{g_{\ve_k}}}=\int_{\Sp^2} \div_{\Phi_k,g_{\ve_k}} \bx \; dvol_{\bar{g_{\ve_k}}}=[2+O(\ve_k^2)] \Area_{g_{\ve_k}}(\Phi_k);
$$
recalling the uniform area bound given in \eqref{eq:AreaBound} we get that there exists $C>0$ such that
\be\label{eq:deltaxArea}
0\leq \frac{1}{C}\leq \delta_{\bx} \Area_{g_{\ve_k}}(\Phi)\leq C < \infty.
\ee
Now let us compute the variation of the Willmore functional with respect to the variation $\bx$:
\be\label{eq:deltaWgek1}
\d_{\bx} W_{g_{\ve_k}}(\Phi_k)=\int_{\Sp^2} <\bx,\bn>_{g_{\ve_k}} \left(\triangle_{\bar{g}_{\ve_k}}H+H |A^\circ|^2+H \Ric(\bn,\bn) \right) dvol_{\bar{g}_{\ve_k}}
\ee
where, of course, all the quantities are computed on $\Phi_k$ and with respect to the metric ${g}_{\ve_k}$. In order to continue the computations, it is useful to rewrite the first variation of $W$ in divergence form. Up to a reparametrization we can assume that $\Phi_k$ are conformal, so  that the following identity holds (see Theorem 2.1 in \cite{MR2}) 
\be\label{eq:WDiv}
\left[ \triangle_{\bar{g_{\ve_k}}} H\ \vec{n}+  \vec{H} |A^\circ|^2 - R^\perp _\Phi (T\Phi)\right] dvol_{\bar{g}_{\ve_k}} = D^*\left[\nabla H \vec{n}-\frac{H}{2} D \vec{n} +\frac{H}{2} \star_{g_{\ve_k}} (\vec{n}\wedge D^\perp \vec{n})\right]\quad
\ee
where $\vec{H}=H\bn$ is the mean curvature vector of the immersion $\Phi_k$, $\star_{g_{\ve_k}}$ is the Hodge operator associated to metric $g_{\ve_k}$, $D \cdot:=(\n_{\p_{x_1} \Phi_k}\cdot, \n_{\p_{x_2} \Phi_k}\cdot)$ and $D^\perp\cdot:=(-\n_{\p_{x_2} \Phi_k}\cdot,\n_{\p_{x_1} \Phi_k}\cdot)$ and $D^*$ is an operator acting on couples of vector fields $(\vec{V}_1,\vec{V}_2)$ along $(\Phi_k)_*(T\Sp^2)$ defined as 
$$D^*(\vec{V}_1,\vec{V}_2):= \n_{\p_{x_1} \Phi_k} \vec{V}_1+\n_{\p_{x_2} \Phi_k} \vec{V}_2.$$
Finally $R^\perp_{\Phi_k}(T\Phi_k):=(Riem(\vec{e_1},\vec{e_2})\vec{H})^\perp =\star_{g_{\e_k}}\left( \vec{n}\wedge Riem^h(\vec{e_1},\vec{e_2})\vec{H} \right)$, where $\vec{e}_i=\frac{\p_{x_i} \Phi}{|\p_{x_i} \Phi|}$ for $i=1,2$. 
\\Plugging \eqref{eq:WDiv} into  \eqref{eq:deltaWgek1} and integrating by parts we obtain
\begin{eqnarray}
\d_{\bx} W_{g_{\ve_k}}(\Phi_k)&=&\int_{\Sp^2} <-D\bx, \nabla H \vec{n}-\frac{H}{2} D \vec{n} +\frac{H}{2} \star_{g_{\ve_k}} (\vec{n}\wedge D^\perp \vec{n}) >_{g_{\ve_k}} dvol_{\Sp^2}\nonumber \\
&&+ \int_{\Sp^2}<\bx, R^\perp _\Phi (T\Phi_k)+\vec{H} \Ric(\bn,\bn) >_{g_{\ve_k}} dvol_{\bar{g}_{\ve_k}}.\label{eq:deltaWgek2}
\end{eqnarray}
Since the Riemannian curvature tensor of the metric $g_{\ve_k}$ is of order $O(\ve_k^2)$ and both the curvature terms are linear in $H$, using Schwartz inequality the integral in the second line can be estimated as
\be\label{eq:2line}
\int_{\Sp^2}<\bx, R^\perp _{\Phi_k} (T\Phi_k)+\vec{H} \Ric(\bn,\bn) >_{g_{\ve_k}} dvol_{\bar{g}_{\ve_k}}=O(\ve_k^2) \left(W_{g_{\ve_k}}(\Phi_k) \, \Area_{g_{\ve_k}}(\Phi_k)\right)^{1/2}= O(\ve_k^2).
\ee
The first line of the right hand side of \eqref{eq:deltaWgek1} can be written explicitely as 
\begin{eqnarray}
\int_{\Sp^2}&&<-\p_{x^1}\Phi_k-\vec{\Gamma}^{g_{\ve_k}}_{\a \b} (\p_{x^1}\Phi_k^\a) \Phi^\b\;, \; (\p_{x^1}H) \bn+\frac{H}{2}A_1^j (\p_{x^j} \Phi_k)+\frac{H}{2} A_2^j \star_{g_{\ve_k}} (\vec{n}\wedge \p_{x^j}\Phi_k) >_{g_{\ve_k}} dvol_{\Sp^2} \nonumber \\
&&+ \int_{\Sp^2} <-\p_{x^2}\Phi_k-\vec{\Gamma}^{g_{\ve_k}}_{\a \b} (\p_{x^2}\Phi_k^\a) \Phi^\b\;, \; (\p_{x^2}H) \bn+\frac{H}{2}A_2^j (\p_{x^j} \Phi_k)-\frac{H}{2} A_1^j \star_{g_{\ve_k}} (\vec{n}\wedge \p_{x^j}\Phi_k) >_{g_{\ve_k}} dvol_{\Sp^2}. \label{eq:line1}  
\end{eqnarray} 
Recalling that $\star_{g_{\ve_k}} (\vec{n}\wedge \p_{x^1}\Phi_k)=\p_{x^2}\Phi_k $ and $\star_{g_{\ve_k}} (\vec{n}\wedge \p_{x^2}\Phi_k)=-\p_{x^1}\Phi_k$ we obtain that all terms obtained doing the scalar product with $-\p_{x^1}\Phi_k$ in the first line, and with $-\p_{x^2}\Phi_k$ in the second line simplify and just the terms containing the Christoffel symbols remain; since $\Phi_k\subset B_{\g_{\ve_k}}(0,2)$ and the Christoffel symbols are of order $O(\ve_k^2)$ by \eqref{eq:Gammage}, \eqref{eq:line1} can be written as
\be \label{eq:line1a}
\int_{\Sp^2}-\sum_{i=1}^2<\vec{\Gamma}^{g_{\ve_k}}_{\a \b} (\p_{x^i}\Phi_k^\a) \Phi^\b\;, \; (\p_{x^i}H) \bn> dvol_{\Sp^2} +O(\ve_k^2)\int_{\Sp^2}|H_{\Phi_k,g_{\ve_k}}|\, |A_{\Phi_k,g_{\ve_k}}|\, dvol_{\bar{g}_{\ve_k}} ;
\ee
using Schwartz inequality of course the second summand can be bounded by 
\be\label{eq:line1aa}
O(\ve_k^2) \left(\int_{\Sp^2} |H_{\Phi_k,g_{\ve_k}}|^2 dvol_{\bar{g}_{\ve_k}}  \right)^{1/2} \left(\int_{\Sp^2} |A_{\Phi_k,g_{\ve_k}}|^2 dvol_{\bar{g}_{\ve_k}}  \right)^{1/2}= O(\ve_k^2),
\ee
where we used the Gauss equations, Gauss-Bonnet Theorem and the area bound \eqref{eq:AreaBound} to infer that
$$\int_{\Sp^2} |A_{\Phi_k,g_{\ve_k}}|^2 dvol_{\bar{g}_{\ve_k}} \leq C (W_{g_{e_k}}(\Phi_k)+1)\leq C_1.$$
In order to estimate the first integral of  \eqref{eq:line1a} we integrate by parts the derivative on $H$ and we recall \eqref{eq:Gammage}, obtaining 
\begin{eqnarray}
\int_{\Sp^2}-\sum_{i=1}^2<\vec{\Gamma}^{g_{\ve_k}}_{\a \b} (\p_{x^i}\Phi_k^\a) \Phi^\b\;, \; (\p_{x^i}H) \bn> dvol_{\Sp^2} = O(\ve_k^2) \int_{\Sp^2} (|H_{\Phi_k,g_{\ve_k}}|+|H_{\Phi_k,g_{\ve_k}}| \, |A_{\Phi_k,g_{\ve_k}}|) dvol_{\bar{g}_{\ve_k}} \nonumber \\ 
= O(\ve_k^2) \left(W_{g_{\ve_k}}(\Phi_k)\right)^{1/2} \left[ \left(\Area_{g_{\ve_k}}(\Phi_k)\right)^{1/2}+\left(\int_{\Sp^2} |A_{\Phi_k,g_{\ve_k}}|^2 dvol_{\bar{g}_{\ve_k}}\right)^{1/2} \right]=O(\ve_k^2).\label{eq:line1b}
\end{eqnarray}
Collecting  \eqref{eq:deltaWgek2},  \eqref{eq:2line},  \eqref{eq:line1}, \eqref{eq:line1a}, \eqref{eq:line1aa} and  \eqref{eq:line1b} we obtain that
$$\d_{\bx} W_{g_{\ve_k}}(\Phi_k)=O(\ve_k^2).$$
Combining the last equation with \eqref{eq:deltaxArea} and \eqref{eq:deltaWArea} we obtain that $\lambda_k=O(\ve_k^2)$ as desired.
\end{proof}

\section{The blow up analysis and the proof of the main theorem} \label{Sec:3}
\subsection{Existence of just one bubble and convergence} \label{SubSec:3.1}
\begin{lem}\label{lem:reparametrize}
Let $g_{\ve_k}$ be the metrics on $\R^3$ defined in \eqref{def:ge} having the expression \eqref{eq:expge} and let $(\Phi_k)$ be area-constrained Willmore immersions of $\Sp^2$ into $(\R^3, g_{\ve_k})$ satisfying \eqref{eq:RescQuant}; without loss of generality we can assume $\Phi_k$ to be conformal with respect to the euclidean metric $g_0$. Up to a rotation in the domain we can also assume that, for every $k \in \N$, the north pole $N\in \Sp^2$ is the maximum point of the quantity $|\n \Phi_k|^2+|\n^2 \Phi_k|$:
$$\mu_k:=|\n \Phi_k|_h^2(N)+|\n^2 \Phi_k|_h(N)=\max_{\Sp^2}|\n \Phi_k|_h^2+|\n^2 \Phi_k|_h, $$
where $h$ is the  standard round metric of $\Sp^2$ of constant Gauss curvature equal to one and $|\nabla \Phi_k|_h, |\nabla^2\Phi_k|_h$ are the norms evaluated in the $h$ metric.

Called $S\in \Sp^2$ the south pole and  $P:\Sp^2\setminus \{S\}\to \R^2$ the stereographic projection, consider the new parametrizations $\tilde{\Phi}_k$, in the sequel simply denoted with $\Phi_k$, defined by
$$
\tilde{\Phi}_k\left(P^{-1}(z)\right):=\Phi_k\left(P^{-1}\left(\frac{z}{\mu_k^{1/2}}\right)\right), \quad \forall z \in \R^2.
$$ 
Then $\tilde{\Phi}_k$, a priori just defined on $\Sp^2\setminus\{S\}$, extend to  smooth conformal immersions of $\Sp^2$ into $(\R^3,g_0)$ and  converge to a conformal parametrization of a round sphere  in $C^l(\Sp^2,h)$-norm, for every $l \in \N$.   
\end{lem}

\begin{proof}

\emph{Step a:} there exists a smooth conformal parametrization $\Phi_\infty:\Sp^2 \to (\R^3,g_0)$ of a round sphere in $\R^3$ endowed with the euclidean metric $g_0$ such that, up to subsequences, $\tilde{\Phi}_k\to \Phi_\infty$ in $C^l_{loc}(\Sp^2\setminus \{S\})$-norm, for every $l\in \N$.
\\

\noindent
Denote by  $u_k$ the conformal factor associated to $\tilde{\Phi}_k$, i.e.
$$\tilde{\Phi}_k^*(g_0) =e^{2u_k} h \quad,$$
where $g_0$ is the euclidean metric in $\R^3$. Observe that, by construction,  for any  compact subset  of the form
$$K:=\Sp^2\setminus B_\delta^h(S) \quad  \text{for some } \delta>0\quad ,$$
there holds
\begin{equation}
\label{est1}
\sup_{k \in \N} \sup_K  \left( |\n \tilde{\Phi}_k|_h^2+|\n^2 \tilde{\Phi}_k|_h \right) < \infty \quad.
\end{equation}
Then, for every compact  there exists a constant $C_K$ depending just on  $K$ such that
 for every $x_0 \in K$ and every $\rho\in \left(0, \frac{dist(K,S)}{2}\right)$ it holds 
 $$\sup_{k\in \N} \sup_{B^h_{\rho}(x_0)}|\nabla^2 \tilde{\Phi}_k|^2\leq C_K\quad ,$$
where $B^h_{\rho}(x_0)$ is the ball of center $x_0$ and radius $\rho$ in the metric $h$. By the conformal invariance of the Dirichlet energy, called $\pi_{\tilde{\bn}_k}$ the projection  on the normal space to $\tilde{\Phi}_k$, we infer that for every $\varepsilon_0>0$ there exists $\rho_{\varepsilon_0,K}>0$ (small enough) depending just on $K$ and on $\varepsilon_0$ but not on  $k\in \N$ such that for every $\rho\in (0,\rho_{\varepsilon_0,K})$ and $x_0 \in K$ it holds
\begin{eqnarray}
\int_{B^h_\rho(x_0)} |\nabla \tilde{\bn}_k|^2_{\tilde{\Phi}_k^*(g_0) } dvol_{\tilde{\Phi}_k^*(g_0) } &=& \int_{B^h_\rho(x_0)} |\nabla \tilde{\bn}_k|^2_{h} dvol_{h }= \int_{B^h_\rho(x_0)} |\pi_{\tilde{\bn}_k} (\nabla^2 \tilde{\Phi}_k)|^2_{h} \, dvol_h\nonumber \\ 
&\leq&   \int_{B^h_\rho(x_0)} |\nabla^2 \tilde{\Phi}_k|^2_{h}  dvol_{h }\leq C_K \rho^2 \leq \varepsilon_0\quad .\label{est1a}
\end{eqnarray}
Taking $\varepsilon_0 \leq \frac{8\pi}{3}$, for any $x_0\in K$ and $\rho< \rho_{\varepsilon_0,K}$ we can apply the H\'elein moving frame method based on Chern construction of conformal coordinates (for more details see \cite{RivPCMI}, Section 3) and infer that, up to a  reparametrization of $\tilde{\Phi}_k$  on $B_{\rho}(x_0)$, called $\bar{u}_k$  the mean value of $u_k$ on $B^h_\rho(x_0)$, it holds
$$\|u_k-\bar{u}_k\|_{L^\infty(B^h_\rho(x_0))}\leq \tilde{C},$$
for some $\tilde{C}>0$ independent of $k\in \N$. Covering  $K$ by finitely many balls as above, the  connectedness of $K$ implies that any two balls of the finite covering are connected by a chain of balls of the same  covering and therefore there exists constants $c_{k,K} \in \R, k \in \N,$ such that
\begin{equation}\label{eq:ck}
\sup_{k\in \N}\|u_k-c_{k,K}\|_{L^\infty(K)}< \infty.
\end{equation}  
Observe that $\sup_{k \in \N}  c_{k,K} <+ \infty$; indeed, if $\limsup_k c_{k,K}=+\infty$ then $\limsup_k \Area(\tilde{\Phi}_k(K))=+\infty$ contradicting the area bound \eqref{eq:AreaBound} (here we use that $K$  has positive $h$-volume). Now let us consider separately the case $\sup_{k}|c_{k,K}|<\infty$ and $\liminf_{k} c_{k,K}=-\infty$ starting from the former.
\\

\emph{Case 1}: $\sup_{k}|c_{k,K}|<\infty$.  Estimate \eqref{eq:ck} yields a uniform bound on the conformal factors $u_k$ on the subset $K$. Since by assumption the immersions $\tilde{\Phi}_k$  are area-constrained Willmore immersions satisfying \eqref{est1a}, then by $\varepsilon$-regularity ($\varepsilon$-regularity for Willmore immersions was first proved by Kuwert and Sch\"atzle in \cite{KS1}. Here we use the $\varepsilon$-regularity theorem proved by Rivi\`ere (see Theorem I.5 in \cite{Riv1}; see also Theorem I.1 in \cite{BRQ}); to this aim observe that the $\varepsilon$-regularity theorem was stated for \emph{Willmore immersions}, but the proof can be repeated  verbatim to \emph{area-constrained Willmore immersions in metric $g_{\e_k}$}: indeed the Lagrange multiplier $\lambda \vec{H}$ and the Riemannian terms are lower order terms that can be absorbed in the already present error terms $\vec{g}_1,\vec{g}_2$ in the proof of  Theorem I.5 at pp. 24-26 in \cite{Riv1}. Of course $\varepsilon$-regularity is a consequence of the ellipticity of the equation.) we infer that for every $l\in \N$ there exists $C_{l}$ such that  
$$|e^{-l\, u_k} \nabla^l \tilde{\Phi}_k | _{L^\infty\left(B^h_{\rho/2}(x_0)\right)}\leq C_l \left( \int_{B^h_\rho(x_0)} |\nabla \tilde{\vec{n}}_k|^2_h  dvol_h +1 \right)^{\frac{1}{2}}\leq \hat{C}_l$$
and therefore, by the assumed uniform bound on $|u_k|$ and by covering $K$ by finitely many balls we get that 
\be\label{eq:epsReg}
\sup_{k \in \N} |\nabla^l \tilde{\Phi}_k | _{L^\infty(K)} < \infty \quad \forall l \in \N \quad.
 \ee
By Arzel\'a-Ascoli Theorem and by the estimate on the Lagrange multipliers given in Lemma \ref{lem:lagrOek}, up to subsequences, the maps $\tilde{\Phi}_k$ converge in $C^l(K)$ norm, for every $l\in \N$, to a limit Willmore immersion $\tilde{\Phi}_\infty$ of $K$ into $(\R^3,g_0)$; repeating the above argument to $K=\Sp^2\setminus B_\delta^h(S)$, for every $\delta>0$, we get that, up to subsequences, the maps $\tilde{\Phi}_k$ converge in $C^l_{loc}(\Sp^2\setminus\{S\})$ norm, for every $l\in \N$, to a limit Willmore immersion $\Phi_\infty:\Sp^2\setminus\{S\}\to \R^3$ is a smooth  Willmore conformal immersion with finite area and $L^2$-bounded second fundamental form, therefore by Lemma A.5 in \cite{Riv2} (let us mention that this result  was already present in \cite{MS}; see also \cite{KL}) the map $\Phi_\infty$ can be extended up to the south pole $S$ to a possibly branched immersion; i.e. the south pole $S$ is a possible branch point for  $\Phi_\infty$ and the following expansion around $S$ holds
\be\label{eq:branch}
(C-o(1)) |z|^{n-1} \leq \left| \frac{\partial \Phi_\infty}{\partial z}\right| \leq (C+o(1)) |z|^{n-1} 
\ee
where $z$ is a complex coordinate around the south pole and $n-1$ is the branching order. We claim that the  branching order is $0$, or in other words that $\Phi_\infty$ is unbranched; indeed, by the strong convergence of $\tilde{\Phi}_k$ to ${\Phi}_\infty$ and the smooth convergence of $g_{\e_k}$ to the euclidean metric $g_0$ we have that 
\be\label{eq:Phiinf<8pi}
W_{g_0}(\Phi_\infty)\leq \liminf_k W_{g_{\e_k}} (\tilde{\Phi}_k)<8\pi;
\ee
therefore, by the Li-Yau inequality \cite{LY}, we get that $n-1=0$, i.e. $\Phi_\infty$ is an immersion also at the south pole $S$. Since $\Phi_\infty$ is a smooth Willmore immersion of $\Sp^2$ into $\R^3$ with energy less than $8\pi$, by the classification of Willmore spheres by Bryant \cite{Bry}, $\Phi_\infty$ is a smooth conformal parametrization of a  round sphere in $\R^3$.
\\

\emph{Case 2}: $\liminf_k c_{k,K}=-\infty$, can not happen. In this case, up to subsequences, we have that $\tilde{\Phi}_k(K)\to \bar{x}\in M$ in Hausdorff distance sense. Consider then the rescaled immersions 
\be\label{eq:hatphik}
\hat{\Phi}_{k}:=e^{-c_{k,K}} \tilde{\Phi}_k
\ee
of $K$ 
and observe that by construction  $\sup_k |\hat{u}_{k,K}|<\infty$, where
  $\hat{u}_{k,K}$ is the conformal factor of $\hat{\Phi}_k$. Moreover, since the integrals appearing in \eqref{est1a} are invariant under rescaling, estimate \eqref{est1a} holds for $\hat{\Phi}_k$ as well. Therefore, up to a diagonal extraction, $\hat{\Phi}_k\to \Phi_\infty$ in  $C^l_{loc}(\Sp^2\setminus \{S\})$-norm. In particular $\tilde{\Phi}_k\to 0$ in  $C^2_{loc}(\Sp^2\setminus \{S\})$-norm, which contradicts the fact that 
$$|\n \tilde{\Phi}_k|_h^2(N)+|\n^2 \tilde{\Phi}_k|_h(N)=1.$$  
  


\emph{Step b}: $\tilde{\Phi}_k\to \Phi_\infty$ in $C^l(\Sp^2)$, for every $l\in \N$; namely the convergence of \emph{Step a} is \emph{on the whole} $\Sp^2$. 
\\Observe that if there exists $\bar{\rho}>0$ such that $\sup_k \sup_{B^h_{\bar{\rho}}(S)} |\nabla \tilde{\Phi}_k|^2+  |\nabla^2 \tilde{\Phi}_k|<\infty$, then in \emph{Step a} we can choose as compact subset $K$ the whole $\Sp^2$ and the claim of  \emph{Step b} follows by the same arguments of \emph{Step a}. So assume by contradiction that there exists a sequence $\rho_k\downarrow 0$ such that, called 
$$\bar{\mu}_k:= \sup_{B^h_{\rho_k}(\bar{x})}  |\nabla \tilde{\Phi}_k|^2+  |\nabla^2 \tilde{\Phi}_k|,$$
 one has 
 $$\limsup_k \bar{\mu}_k=+\infty.$$ 
 By a small rotation in the domain $\Sp^2$ we can assume that, for every $k \in \N$, the maximum of  $|\nabla \tilde{\Phi}_k|^2+  |\nabla^2 \tilde{\Phi}_k|$ on $B^h_{\rho_k}(S)$ is attained at the south pole $S$ and that, up to subsequences in $k$, it holds
\be\label{eq:poloS}
\lim_k \bar{\mu}_k :=\lim_k |\nabla \tilde{\Phi}_k|^2(S)+  |\nabla^2 \tilde{\Phi}_k|(S) =+\infty.
\ee
Analogously to above, called $P_N:\Sp^2\setminus\{N\}\to \R^2$ the stereographic projection centered at the north pole $N$,  we consider the reparametrized immersions 
$$
\bar{\Phi}_k\left(P^{-1}_N(z)\right):=\tilde{\Phi}_k\left(P^{-1}_N\left(\frac{z}{\bar{\mu}_k^{1/2}}\right)\right).
$$ 
Observe that, in this way, the compact subsets $K$'s  considered above are shrinking   towards the north pole $N$ and, by the arguments above, their $\bar{\Phi}_k$-images are converging to a round sphere; repeating the arguments above to compact subsets this time containing the south pole $S$ and  avoiding the north pole $N$ we infer that, up to subsequences, $\bar{\Phi}_k$ (or a further rescaled of it) converges smoothly, away the north pole $N$, to a round sphere; namely a second bubble. Combining the bubble formed in \emph{Step a} and this second bubble, since each bubble contributes $4\pi$ of Willmore energy, we infer that
\be\label{eq:2bubble}
\limsup_k W_{g_{\e_k}}({\Phi}_k)\geq 8 \pi \quad,
\ee
contradicting the assumption \eqref{eq:RescQuant}. This concludes the proof of the \emph{Step b} and of the lemma .
\\
\end{proof}

\subsection{Expansion of the equation}
Recalling that  $\Phi_k:\Sp^2 \hookrightarrow (\R^3,g_{\ve_k})$ is a  smooth  immersion satisfying the area-constrained Willmore equation in metric $g_{\ve_k}$, and that $g_{\ve_k}$ smoothly converge to the euclidean metric $g_0$, in the present  section we expand this differential equation with respect to $\ve_k$. Without loss of generality we can assume that $\Phi_k$ is conformal with respect to the metric $g_{\ve_k}$. We will see that curvature terms appear at $\ve_k^2$ order while the derivatives of the curvature appear at  $\ve_k^3$ order. 
\\ \emph{From now on, in order to make the notation a bit lighter, we replace  $\ve_k$ by $\ve$.}
\\Recall that the area-constrained Willmore equation in metric $g_{\ve}$ has the following form
\be\label{eq:AConsWill}
\triangle_{\bar{g}_{\ve}} H_{\ve} +H_{\ve} \vert A^\circ_{{\ve}}\vert_{\bar{g}_{\ve}}^2+ \Ric_{g_{\ve}}(\vec{n}_{\ve}, \vec{n}_{\ve}) H_{\ve} = \lambda_{\ve} H_{\ve} \quad.
\ee
Since $\triangle_{\bar{g}_{\ve}} =\frac{2}{\vert \nabla \Phi_{\ve} \vert_\gve^2}\Delta$, where $\Delta$ is the flat laplaciain in $\R^2$, multiplying  equation \eqref{eq:AConsWill} by $\frac{\vert \nabla \Phi_{\ve} \vert_\gve^2}{2}$, we get
\be
\label{em}
\Delta  H_{\ve}  + \frac{\vert \nabla \Phi_{\ve} \vert_{g_{\ve}}^2}{2}H_{\ve} \vert A^\circ_{{\ve}}\vert_{\bar{g}_{\ve}}^2+\frac{\vert \nabla \Phi_{\ve} \vert_\gve^2}{2}H_{\ve} \Ric_{g_{\ve}}(\vec{n}_{\ve}, \vec{n}_{\ve})  = \lambda_{\ve} \frac{\vert \nabla \Phi_{\ve} \vert_\gve^2}{2}H_{\ve}\quad.
\ee
First of all, recalling that $H_{\ve} = \frac{g_{\ve}\left( \triangle_{\bar{g}_{\ve}} \Phi_{\ve} , \vec{n}_{\ve}\right)}{2}$, we expand $H_{\ve}$ as
\be\label{eq:He1}
H_{\ve} =\frac{1}{\vert \nabla \Phi_{\ve} \vert_{g_{\ve}}^2} ({g_{\ve}})_{\a \b}  \triangle  \Phi_{\ve}^\a \sqrt{\vert g_{\ve}\vert} g_{\ve}^{\b \gamma}  (\vec{\nu}_{{\ve}})_ {\gamma} = \frac{\sqrt{\vert g_{\ve}\vert}}{\vert \nabla \Phi_{\ve} \vert_{g_{\ve}}^2}  \triangle  \Phi_{\ve}^\a  \vec{\nu}_{{\ve}\a}
\ee
where $\vec{\nu}_{\ve}$ is the inward pointing unit normal with respect to $g_0$. Using (\ref{eq:expge}) and (\ref{eq:detge}), we get 
\begin{eqnarray}
\vert \nabla \Phi_{\ve} \vert_\gve^2&=& \vert \nabla \Phi_{\ve} \vert^2 + \frac{{\ve}^2}{3} R_{\a \b \g \eta}(p_k) \Phi_{\ve}^\b \Phi_{\ve}^\g \langle \nabla \Phi_{\ve}^\a,  \nabla \Phi_{\ve}^\eta \rangle \nonumber \\
&& + \frac{{\ve}^3}{6} R_{\a \b  \g \eta ,\mu }(p_k) \Phi_{\ve}^\b \Phi_{\ve}^\g \Phi_{\ve}^\mu \langle \nabla \Phi_{\ve}^\a,  \nabla \Phi_{\ve}^\eta \rangle + O({\ve}^4) \quad, \nonumber
\end{eqnarray}
so that
\begin{eqnarray}
\frac{1}{\vert \nabla \Phi_{\ve} \vert_\gve^2}&=&\frac{1}{\vert \nabla \Phi_{\ve}\vert^2}   \Big(1 - \frac{{\ve}^2}{3\vert \nabla \Phi_{\ve}\vert^2} R_{\a \b \g \eta}(p_k) \Phi_{\ve}^\b \Phi_{\ve}^\g \langle \nabla \Phi_{\ve}^\a,  \nabla \Phi_{\ve}^\eta \rangle \nonumber \\
&&  -\frac{{\ve}^3}{6 \vert \nabla \Phi_{\ve}\vert^2} R_{\a \b  \g \eta ,\mu }(p_k) \Phi_{\ve}^\b \Phi_{\ve}^\g \Phi_{\ve}^\mu \langle \nabla \Phi_{\ve}^\a,  \nabla \Phi_{\ve}^\eta \rangle + O({\ve}^4)  \Big)\quad, \label{eq:1/nablaPhi2}
\end{eqnarray}
moreover
\be\label{eq:|ge|}
\sqrt{\vert \gve\vert}  = 1 - \frac{{\ve}^2}{6} \Ric_{\a \b}(p_k) \Phi_{\ve}^\a \Phi_{\ve}^\b- \frac{{\ve}^3}{6} \Ric_{\a \b,\gamma}(p_k) \Phi_{\ve}^\a \Phi_{\ve}^\b \Phi_{\ve}^\g + O({\ve}^4).
\ee
Combining \eqref{eq:He1} with  \eqref{eq:1/nablaPhi2} and \eqref{eq:|ge|} we can write
\be
\label{exp1}
 H_{\ve} = \frac{ \triangle  \Phi_{\ve}^\a \vec{\nu}_{{\ve} \a} }{\vert \nabla \Phi_{\ve}\vert^2}\left(1 + {\ve}^2 S_{\ve} + {\ve}^3 T_{\ve} + O({\ve}^4)\right) \quad,
 \ee
 where
 $$
 S_{\ve}:=  - \frac{1}{3\vert \nabla \Phi_{\ve}\vert^2} R_{\a \b \g \eta}(p_k) \Phi_{\ve}^\b \Phi_{\ve}^\g \langle \nabla \Phi_{\ve}^\a,  \nabla \Phi_{\ve}^\eta \rangle- \frac{1}{6} Ric_{\a \b}(p_k) \Phi_{\ve}^\a \Phi_{\ve}^\b 
 $$
  and
  $$T_{\ve}:=   - \frac{1}{6\vert \nabla \Phi_{\ve}\vert^2} R_{\a \b \g \eta,\mu}(p_k) \Phi_{\ve}^\b \Phi_{\ve}^\g \Phi_{\ve}^\mu\langle \nabla \Phi_{\ve}^\a,  \nabla \Phi_{\ve}^\eta \rangle - \frac{1}{6} Ric_{\a\b,\g}(p_k) \Phi_{\ve}^\a \Phi_{\ve}^\b \Phi_{\ve}^\g .
$$
The combination of  \eqref{eq:|ge|} and   \eqref{exp1}  gives
\be
\label{exp3}
\Ric_{g_{\ve}}(\vec{n}_{\ve}, \vec{n}_{\ve}) H_{\ve}= {\ve}^2  \frac{ \triangle  \Phi_{\ve}^\a  \vec{\nu}_{{\ve} \a} }{\vert \nabla \Phi_{\ve}\vert^2}  \Ric_{g}(p_k)(\vec{\nu}_{{\ve} }, \vec{\nu}_{{\ve} })  + O({\ve}^4). 
\ee
Finally, using  (\ref{exp1}), (\ref{exp3}) and (\ref{eq:lambdae2}), we expand (\ref{em}) up to ${\ve}^2$-order (the term  $H_{\ve} \vert A^\circ_{{\ve}}\vert_{\bar{g}_{\ve}}^2$  will be expanded in the next subsection) as follows
\begin{equation}
\label{expeq}
\begin{split}
& \Delta H_{\ve} + \frac{\vert \nabla \Phi_{\ve} \vert_\gve^2}{2} H_{\ve} \vert A _{{\ve}} ^\circ \vert_{\bar{g}_{\ve}}^2+ \frac{\vert \nabla \Phi_{\ve} \vert_\gve^2}{2} H_{\ve}  \Ric_{g_{\ve}}(\vec{n}_{\ve}, \vec{n}_{\ve})  - \lambda_{\ve} H_{\ve} \frac{\vert \nabla \Phi_{\ve} \vert_\gve^2}{2} =\\
& \Delta \left( \frac{ \Delta  \Phi_{\ve}^\a  \vec{\nu}_{{\ve} \a}}{\vert \nabla \Phi_\ve\vert^2} \right) + {\ve}^2 \left(\Delta \left( \frac{ \Delta  \Phi_{\ve}^\a  \vec{\nu}_{{\ve} \a}}{\vert \nabla \Phi_\ve \vert^2}  \right) S_{\ve} +2 \left\langle\nabla\left( \frac{ \Delta  \Phi_{\ve}^\a  \vec{\nu}_{{\ve} \a}}{\vert \nabla \Phi_\ve\vert^2}\right)   ,\nabla S_{\ve} \right\rangle + \frac{ \Delta  \Phi_{\ve}^\a  \vec{\nu}_{{\ve} \a}}{\vert \nabla \Phi_\ve\vert^2}  \Delta S_{\ve} \right)\\
&+ \frac{\vert \nabla \Phi_{\ve} \vert_\gve^2}{2} H_{\ve} \vert A^\circ_{{\ve}}\vert_{\bar{g}_{\ve}}^2 +  \frac{{\ve}^2}{2}   \Delta  \Phi_{\ve}^\a  \vec{\nu}_{{\ve} \a}  \Ric_{g}(p)(\vec{\nu}_{{\ve} }, \vec{\nu}_{{\ve} })  - \frac{\lambda_{\ve}}{2}  \Delta  \Phi_{\ve}^\a  \vec{\nu}_{{\ve} \a} +  o({\ve}^2) . 
\end{split}
\end{equation}

\subsection{Approximated solutions to the area-constrained Willmore equation}
\label{approx}
In this section we solve (\ref{expeq}) up to the $\ve^2$ order. For this let  $\omega$ be the inverse of the stereographic projection with respect to the north pole and notice that  $\omega$ is  a solution of the equation when $\ve=0$. We make the ansatz of looking  for a solution up to the order $\ve^2$ of the form $\omega + \ve^2 \rho$, for some function $\rho$. Since $ \vert A^\circ\vert^2=0$ for $\omega$, it is clear that
\be
\label{exp2}
 H_\ve \vert A^\circ_{\ve}\vert_{\bar{g}_{\ve}}^2 =O(\ve^4);
 \ee
in particular, since for our arguments it is enough to expand the equation up to $\ve^3$ order, this term will never play a role and therefore it will be neglected.
\\Observing that $\frac{\Delta \omega^\a \omega_\a}{|\nabla \omega|^2}\equiv-1$, equation \eqref{expeq} implies that $\rho$ must solve 
\be
\label{ee2}
\begin{split}
 L_{\omega}(\rho) &= \Delta\left( \frac{1}{3\vert \nabla \omega\vert^2} R_{\a \b \g \mu}(p_k) \omega^\b \omega^\g \langle \nabla \omega^\a,\nabla \omega^\mu \rangle+ \frac{1}{6} Ric_{\a \b}(p_k) \omega^\a \omega^\b  \right) \\
 &\quad -  \frac{\vert \nabla \omega\vert^2}{2}Ric_{\a \b}(p_k) \omega^\a \omega^\b  + \frac{\lambda_{\ve}}{2\ve^2} \vert \nabla \omega\vert^2\quad , 
 \end{split}
 \ee
where $L_\omega$ is the linearized  Willmore operator  at $\omega$, see Appendix \ref{aL} for more details. Using the identity
\be
\label{romega}
 \langle \nabla \omega^\a,\nabla \omega^\b \rangle = (\delta_{\a \b}-\omega^\a \omega^\b) \frac{\vert \nabla\omega \vert^2}{2},
 \ee 
equation (\ref{ee2}) reduces to 
\be
\begin{split}
 L_{\omega}(\rho) &=  \frac{1}{3}\Delta\left( Ric_{\a \b}(p_k) \omega^\a \omega^\b  \right)-  \frac{\vert \nabla \omega\vert^2}{2}Ric_{\a \b}(p_k) \omega^\a \omega^\b  + \frac{\lambda_{\ve}}{2\ve^2} \vert \nabla \omega\vert^2  \\
 &= \left(-  Ric_{\a \b}(p_k) \omega^\a \omega^\b + \left(\frac{\lambda_{\ve}}{2\ve^2} + \frac{Scal(p_k)}{3}\right) \right)\vert \nabla \omega\vert^2.
 \end{split}
 \ee
Hence, we easily check that 
\be
\label{defro}
\rho_{\ve}= \frac{1}{3}\Ric_{\a \b}(p_k) \omega^\b + \frac{\lambda_{\ve}}{\ve^2} f(r) \omega 
\ee
with
$$f(r)= \frac{ r^2\ln\left(\frac{r^2}{1+r^2}\right)-1-\ln\left(1+r^2\right)}{1+r^2}.$$
where $r^2=x^2 +y^2$, is the desired function. Moreover it is not difficult to  check that this perturbed $\omega$ satisfies the conformal conditions up to $\ve^2$ order,  that is to say

\be
\label{eqr0}
\left\{\begin{array}{l}g_{\ve}( (\omega +{\ve}^2 \rho_{\ve})_x,(\omega +{\ve}^2 \rho_{\ve})_x)  -   g_{\ve}( (\omega +{\ve}^2 \rho_{\ve})_y,(\omega +{\ve}^2 \rho_{\ve})_y) = O({\ve}^3) \\ g_{\ve}( (\omega +{\ve}^2 \rho_{\ve})_x,(\omega +{\ve}^2 \rho_{\ve})_y) = O({\ve}^3) \end{array}\right. ;
\ee
a way to prove it is to  use  the expansion of the metric with the fact that   in dimension $3$ one has 
\be\nonumber
R_{\a \b \g \mu} =(g_{\a \g }\mathrm{Ric}_{\b \mu }-g_{\a \mu}\mathrm{Ric}_{\b \g} + g_{\b \mu}\mathrm{Ric}_{\a \g} - g_{\b \g}\mathrm{Ric}_{\a \mu}) +\frac{\hbox{Scal}}{2} (g_{\a \mu}g_{\b \g}-g_{\a \g }g_{\b \mu}).
\ee

\subsection{Proof of Theorem \ref{thm2}} \label{subsec:3.4}

Let us briefly recall the setting. Let $\Phi_k : \Sp^2 \hookrightarrow (M,g)$ be  conformal Willmore immersions satisfying 
\begin{eqnarray} 
\ve&:=&\diam_g(\Phi_k(\Sp^2)) \to 0 \label{eq:diamto0Pf},\\
W_g(\Phi_k)&:=& \int_{\Sp^2} |H_{g,\Phi_k}|^2 dvol_{\bar{g}_k}\leq 8 \pi-2\delta,\quad \text{for some } \delta>0 \text{ independent of } k \label{eq:HpWgPf}.
\end{eqnarray}
Thanks to Lemma \ref{lem:lagrOek}, we associate to $\Phi_k$ the new immersion  $\Phi^\ve : \Sp^2 \hookrightarrow (\R^3, g_{\ve})$,  where $
g_{\ve}(y)(u,v):=g(\ve y)(\ve^{-1}u, \ve^{-1} v)$, 
which satisfies the area-constrained Willmore equation
\be\label{eq:AreaConstrWkPf}
\triangle_{\bar{g}_{\ve}}H_{g_{\ve},\Phi^\eps}+H_{g_{\ve},\Phi^\eps} \; |A^\circ_{g_{\ve},\Phi^\eps}|_{\bar{g}_{\ve}}^2+H_{g_{\ve},\Phi^\eps} \Ric_{g_{\ve}}(\bn_{g_{\ve},\Phi^\eps}, \bn_{g_{\ve},\Phi^\eps})=\lambda_\eps H_{g_{\ve},\Phi^\eps}
\ee
with  $\lambda_\eps =O(\ve^2)$. 
Moreover by Lemma \ref{lem:reparametrize} we know that,  up to conformal reparametrizations and up to subsequences,  we have
$$\Phi^{\ve} \rightarrow \Phi \hbox{ in } C^2(\Sp^2) $$
where $\Phi$ is a conformal diffeomorphism of $\Sp^2$. Clearly, up to reparametrizing our sequence, we can assume that $\Phi=Id$. In the following we perform all the computations in the chart given by the stereographic projection (which is conformal); we denote by $\omega$ the inverse of the stereographic projection.\\

Before proceeding with the proof, we need to make a small adjustment to the immersions. We claim that there exist $a^{\ve}\in \R^2$, $b^{\ve}\in \R^2$, $R^\eps\in SO(3)$  and $z^{\ve}\in \C$  satisfying  
\be 
\label{pest}
a^{\ve}=o(1), \quad  b^{\ve}=o(1), \quad \vert Id -R^\eps\vert =o(1), \quad z^{\ve}=o(1) \quad  \quad,
\ee
 such that,  up to replacing $\Phi^{\ve}$ by $\Phi^{\ve}(a^{\ve} + z^{\ve}\, . \,)$, and  $ \Omega^{\ve} = \omega^{\ve} +\ve^2 \rho^{\ve}$, where $\rho^{\ve}$ is given by (\ref{defro}),  by $R^\eps [ \omega( \cdot + b^{\ve}) + \ve^2 \rho^{\ve}(\cdot+b_\ve)] $ we get 
\be 
\label{pest2}
\begin{split}
\vert \nabla \Phi^{\ve} \vert  \hbox{ and }\vert \nabla \Omega^\eps \vert \hbox{ are maximal at } 0,& \quad  Vect\{ \Phi^{\ve}_x(0) ,  \Phi^{\ve}_y(0) \}=Vect\{ \Omega^{\ve}_x(0) ,   \Omega^{\ve}_y(0) \},\quad  \\
&\text{and } \Phi^{\ve}_x(0)= \Omega^{\ve}_x(0). 
\end{split}
\ee
This is a simple consequence of the $C_{loc}^2(\R^2)$ convergence of $\Phi^{\ve}$ to $\omega$. Indeed, we choose first $a^\eps$ and $b^\eps$ such that $\vert \nabla \Phi^{\ve} \vert$ and $\vert \nabla \Omega^\eps \vert $ are maximal at $0 $, then $R^\eps$ such that the tangent plane of $\Phi^\eps$ and $R^\eps \Omega^\eps$ coincide at $0$ and finally we find $z_\eps$ in order to adjust the first derivatives. \\
Therefore from now on we will assume that \eqref{pest2} is satisfied.\\

Now we  prove Theorem \ref{thm2}.  We set
$$\Phi^{\ve}=\Omega^{\ve} +r^{\ve}$$
for some function $r^\ve$ and,  thanks to the computations of Section \ref{approx}, we see that  $r^{\ve} $ satisfies
\be
\label{eqr1}
L_\omega(r^{\ve}) = O\left(\ve^3\right)  +o\left(  \vert \n r^{\ve}  \vert + \vert \n^2 r^{\ve}\vert +\vert \n^3 r^{\ve}\vert +\vert \n^4 r^{\ve}\vert  \right) .
\ee

Moreover, combining \eqref{eqr0} and (\ref{pest2}), we get that 
\be
\label{eqr3}
 g^\ve( \nabla r^{\ve} ,\nabla r^{\ve} )(0) =O(\ve^6).
\ee
Indeed the error terms of  $r^\eps_x(0)$ and  $r^\eps_y(0)$ lie in the plane generated by $\Omega^\eps_x(0)$ and $\Omega^\eps_y(0)$. So it suffices to estimate their projection against $\Omega^\eps_x(0)$ and $\Omega^\eps_y(0)$. But this one vanish up to the $\eps^3$ order thanks to (\ref{eqr0}).  Observe that we also have 
\be
\label{eqr4}
g^\ve (\nabla^2 r^{\ve} , \nabla \omega^{\ve} )(0) =O(\ve^3). 
\ee

{\bf Claim : $\sup_{\R^2}  \vert \n r^{\ve}  \vert + \vert \n^2 r^{\ve}\vert +\vert \n^3 r^{\ve}\vert +\vert \n^4 r^{\ve}\vert  =O(\ve^3)$}.\\

{\it Proof of the Claim:} let us denote $\mu_{\ve} :=  \vert \n r^{\ve}  \vert + \vert \n^2 r^{\ve}\vert +\vert \n^3 r^{\ve}\vert +\vert \n^4 r^{\ve}\vert $ and    assume by contradiction  that $\lim \frac{\ve^3}{\mu_{\ve}} =0$. Up to a reparametrization we can assume that this $\sup$ is achieved  at some  point $z_\ve$ which is confined  in a fixed  compact subset of $\R^2$. In fact, we can do a reparametrization in order to make this  requirement satisfied  before performing the  adjustments of the previous page. Then we set
$$\tilde{r}_{\ve}= \frac{r_{\ve}  - r_{\ve}(0)}{\mu^{\ve}} . $$
 By construction, $\tilde{r}^\ve$ is bounded in $C^4$-norm on every compact subset of $\R^2$ and therefore, by Arzel\'a-Ascoli's Theorem, it converges up to subsequences to a limit function  $\tilde{r}$ in $C^3_{loc}$-topology.   Thanks to (\ref{eqr1}), $\tilde{r}$ is a solution of the linearized equation (\ref{lin})   and, recalling (\ref{eqr3})-(\ref{eqr4}), it  satisfies  (\ref{linc})  with $\nabla \tilde{r} (0) =0$ and $\langle\nabla^2 \tilde{r} , \nabla \omega\rangle(0)=0$. Then, applying  Lemma \ref{coin}, we get that $\nabla \tilde{r} \equiv 0$, which is in contradiction with the fact that  $\vert \n \tilde{r}  \vert + \vert \n^2 \tilde{r} \vert +\vert \n^3 \tilde{r} \vert +\vert \n^4 \tilde{r}\vert =1$ at some point at finite distance. This proves the claim. \hfill$\square$\\

\noindent Mimicking the proof of the claim above, one can prove that  setting 

$$\tilde{r}_{\ve}= \frac{r_{\ve}  - r_{\ve}(0)}{\ve^3}, $$
then, up to  subsequences,  $\tilde{r}_{\ve}$ converges to a function $\tilde{r}$ in $C^{3}_{loc}(\R^2)$ which, using  (\ref{em}), (\ref{exp1}) and (\ref{exp3}),  satisfies the linearized Willmore equation
$$L_\omega(\tilde{r}) = \Delta\left( \frac{1}{6\vert \nabla \omega\vert^2} R_{\a \b \g \mu, \nu}(p_k) \omega^\b \omega^\g \omega^\nu \langle \nabla \omega^\a,  \nabla \omega^\mu \rangle + \frac{1}{6} \Ric_{\a \b, \gamma}(p_k) \omega^\a \omega^\b \omega^\gamma \right).$$
Recalling identity (\ref{romega}), the last equation can be rewritten as follows
$$L_\omega(\tilde{r}) = \Delta\left( \frac{1}{12} \Ric_{\a \b,\g}(p_k) \omega^\a \omega^\b \omega^\g \right).$$
Finally, integrating this relation against the $\omega^\a$, for $\a=1,\dots,3$, which are solutions of the linearized equation, we get
$$\int_{\R^2} \Delta\omega \, \left( \frac{1}{12} \Ric_{\a \b,\g}(p_k) \omega^\a \omega^\b \omega^\g \right) \, dz =0.$$
Let us note that the integration by parts above has been possible thanks to the decay of $\omega$ and its derivatives at infinity. The last identity  gives
$$\int_{\R^2} \left( \Ric_{\a \b,\g}(p_k) \omega^\a \omega^\b \omega^\g \right) \omega \frac{\vert \nabla \omega \vert^2}{2} \, dz =0.$$
Then, by a  change of variable, we get
$$\int_{\Sp^2} \left( \Ric_{\a \b,\g}(p_k)(p_k) y^\a y^\b y^\g \right) y  \, dvol_{h} =0,$$
where $h$ is the standard metric on $\Sp^2$ and $y^\a$ are the position coordinates of $\Sp^2$ in $\R^3$. Finally, using the following relation
$$ \int_{S^2} y^\a y^\b y^\g y^\mu dvol_h =\frac{4\pi}{15}(\delta^{\a \b}\delta^{\mu \g} + \delta^{\a \mu}\delta^{\b \g} +\delta^{\a \g}\delta^{\b \mu}),  $$
and the second Bianchi identity, we obtain
$$\nabla \Scal(\bar{p})=0,$$
Which proves the theorem.\hfill$\square$ 

\appendix
\section{The linearized Willmore operator}
\label{aL}
The aim of this appendix is to derive the linearized Willmore equation and to classify its solution. 
\\The Willmore equation  for a conformal immersion $\Phi$ into $\R^3$ can be written as
$$ W'(\Phi)=\Delta_{\bar{g}}\left( H\right) + H \vert A^\circ \vert^2_{\bar{g}}=0,$$
where $\Delta_{\bar{g}} = \frac{2}{\vert \nabla \Phi\vert^2}\Delta$, $H$ is the mean curvature and $A^\circ$ is the trace-less second fundamental form. 
\\Equivalently one has 
$$H= \frac{1}{2}\langle \Delta_{\bar{g}} \Phi , \vec{\nu}\rangle $$ 
where $\vec{\nu}$ is the inward pointing unit normal of the immersion $\Phi$. Hence,  by multiplying the first equation by $\frac{\vert \nabla \Phi \vert^2}{2}$, we can consider the equivalent  equation
$$ \widetilde{W'}(\Phi)=\Delta H  + \langle \Delta \Phi , \vec{\nu}\rangle \frac{\vert A^\circ \vert^2_{\bar{g}}}{2}=0.$$
Of course any  conformal  parametrization, $\omega$,  of a round sphere is  a solution. Then, expanding $\widetilde{W'}(\omega +t \rho)$ for some function $\rho$ and using the fact that $A^\circ \equiv 0$ for a round sphere, we get 

\be
\label{lin}
L_\omega (\rho) := \delta \widetilde{W}_{\omega}(\rho)= -\Delta \left( \frac{\langle \Delta \rho , \omega\rangle +2 \langle \nabla \omega, \nabla \rho \rangle}{\vert \nabla \omega \vert^2 }\right) =0 .
\ee
Let also consider the linearization of the conformality condition, which gives 
\be
\label{linc}
\left\{\begin{array}{l}\langle \omega_x, \rho_x\rangle- \langle \omega_y, \rho_y\rangle=0
 \\ \langle \omega_x, \rho_y\rangle + \langle \omega_y, \rho_x\rangle=0 \end{array}\right.
\ee
In the following lemma we classify the solutions of the linearized operator following the previous work \cite{Lau1} of the first author concerning the linearized operator for the constant mean curvature equation.  

\begin{lem}
\label{coin} Let $\rho\in \mathring{H}^2( \R^2, \R^3)$\footnote{the pushforword of $H^{2}(S^2)$ on $\R^2$ via the stereographic projection. } be a solution of the linearized equation (\ref{lin})  which satisfies (\ref{linc}) and the additional normalizing conditions 
$$\nabla \rho(0)=0 \quad \text{ and } \quad \langle \nabla^2 \rho, \nabla \omega\rangle(0) =0\quad.$$
Then  $\nabla \rho\equiv 0.$    
\end{lem}
\begin{proof}
First we remark that thanks to the definition of $\mathring{H}^2( \R^2, \R^3)$, we have
$$\frac{\langle \Delta \rho , \omega\rangle +2 \langle \nabla \omega, \nabla \rho \rangle}{\vert \nabla \omega \vert^2 }\in L^2(\R^2) .$$
Hence, using Liouville's theorem, we get  that 
\be\label{eq:Liouville}
\langle \Delta \rho , \omega\rangle +2 \langle \nabla \omega, \nabla \rho \rangle =0.
\ee
Then, thanks to the fact that $(\omega_x,\omega_y,\omega)$ is a basis of $\R^3$ and (\ref{linc}), there exists $a,b,c,d: \R^2 \rightarrow \R$ such that 

\be
\label{dec}
\left\{\begin{array}{l}\rho_x = a \omega_x + b\omega _y + c \omega  \\ \rho_y = -b \omega_x + a \omega _y + d \omega \quad.\end{array}\right.
\ee

Then, plugging (\ref{dec}) into \eqref{eq:Liouville} and using the relation $\rho_{xy}=\rho_{yx}$, we see that $a,b,c,d$ satisfy  the following equations
\begin{eqnarray} 
a_y +b_x&=&d \quad \label{de1} \\
b_y -a_x &=&-c \quad \label{de2}\\
c_y -d_x&=& b \vert \nabla \omega \vert^2 \quad  \nonumber \\
c_x +d_y&=&-a \vert \nabla \omega \vert^2 \quad. \nonumber 
\end{eqnarray}
These equations imply that $a$ and $b$ satisfy
$$\Delta a = - a \vert \nabla \omega\vert^2, \quad \Delta b = - b \vert \nabla \omega\vert^2 . $$
Since $\rho \in \mathring{H}^1( \R^2, \R^3)$, then  $a$ and $b$ can be seen as  functions in $H^1(S^2)$ satisfying $\Delta \alpha = 2 \alpha$ , therefore  $a$ and $b$ are linear combinations of the first non vanishing eigenfunctions of $\Delta_{\Sp^2}$ (see also Lemma C.1 of \cite{Lau1}),  that is to say
 $$a = \sum_{i=0}^{2} a_i \psi_i \hbox{ and } b = \sum_{i=0}^{2} b_i \psi_i $$
 where 
$$\psi_i (x) = \frac{x_i}{(1+\vert x\vert^2)} \hbox{ for } i=1,2 \hbox{ and } \psi_0(x)= \frac{1-\vert x\vert^2}{1+\vert x \vert^2}. $$
Finally using the fact that $\nabla \rho(0) =0$ and $\langle \nabla^2 \rho, \nabla \omega\rangle(0) =0$, (\ref{de1}) and (\ref{de2}), we can conclude that $a\equiv b\equiv c\equiv  d \equiv0$, which proves the lemma.\end{proof}

\end{document}